\documentclass[12pt]{article} 

\usepackage{amsmath}
\usepackage{amsthm}
\usepackage{amsfonts}
\usepackage{mathrsfs}
\usepackage{stmaryrd}
\usepackage{setspace}
\usepackage{fullpage}
\usepackage{amssymb}
\usepackage{breqn}
\usepackage{enumitem}
\usepackage{bbold} 
\usepackage{authblk}
\usepackage{comment}
\usepackage{hyperref}
\usepackage{pgf,tikz}
\usepackage{graphicx}
\usepackage{colonequals}
\usetikzlibrary{decorations.pathreplacing,arrows}
\tikzstyle{vertex}=[circle,draw=black,fill=black,inner sep=0,minimum size=3pt,text=white,font=\footnotesize]

\newtheorem*{thm*}{Theorem}
\newtheorem{thm}{Theorem}

\newtheorem{lemma}[thm]{Lemma}

\newtheorem{corollary}[thm]{Corollary}

\newtheorem*{proposition*}{Proposition}

\newcommand\cH{{\mathcal H}}

\newcommand{\ignore}[1]{}

\title{On Berge-Ramsey problems}
\author{D\'aniel Gerbner\thanks{Research supported by the J\'anos Bolyai Research Fellowship of the Hungarian Academy of Sciences and the National Research, Development and Innovation Office -- NKFIH under the grants K 116769, KH 130371 and SNN 129364.}
\\
\small Alfr\'ed R\'enyi Institute of Mathematics, Hungarian Academy of Sciences\\
\small P.O.B. 127, Budapest H-1364, Hungary.\\
\small \texttt{gerbner@renyi.hu}}

\begin{document}
\maketitle

\begin{abstract} Given a graph $G$, a hypergraph $\cH$ is a Berge copy of $F$ if $V(G)\subset V(\cH)$ and there is a bijection $f:E(G)\rightarrow E(\cH)$ such that for any edge $e$ of $G$ we have $e\subset f(e)$. We study Ramsey problems for Berge copies of graphs, i.e. the smallest number of vertices of a complete $r$-uniform hypergraph, such that if we color the hyperedges with $c$ colors, there is a monochromatic Berge copy of $G$. 

We obtain a couple results regarding these problems. In particular, we determine for which $r$ and $c$ the Ramsey number can be super-linear. We also show a new way to obtain lower bounds, and improve the general lower bounds by a large margin. In the specific case $G=K_n$ and $r=2c-1$, we obtain an upper bound that is sharp besides a constant term, improving earlier results.

\end{abstract}

\section{Introduction}

Gerbner and Palmer, \cite{gp1}, extending the definition of hypergraph cycles due to Berge, introduced the so-called Berge hypergraphs. Given a graph $G$ and a hypergraph $\cH$, we say that $\cH$ is a \textit{Berge copy of $F$} (in short: \textit{Berge-$F$}) if $V(G)\subset V(\cH)$ and there is a bijection $f:E(G)\rightarrow E(\cH)$ such that for any edge $e$ of $G$ we have $e\subset f(e)$. In other words, we can extend edges of $G$ to obtain $\cH$, or we can shrink hyperedges of $\cH$ to obtain a copy of $G$. Observe that a graph $G$ may have several non-isomorphic Berge copies, and a hypergraph $\cH$ may be Berge copy of several different graphs. 

For a graph $G$, we denote by $B^rG$ the family of $r$-uniform Berge copies of $G$.
As we will often talk about graphs and hypergraphs in the same proofs, we will always refer to graph edges as \emph{edges} and hypergraph edges as \emph{hyperedges} or $r$-edges (if they have size $r$), to help distinguish them.

Tur\'an-type problems for Berge hypergraphs have attracted a lot of researchers, see e.g. \cite{fkl,gmp,gmv,gp2,GMTthreshold,gykl,gyl}. Recently, the study or Ramsey problems for Berge hypergraphs has been initiated independently by three groups of researchers \cite{agy,gmov,STWZ} (for Berge cycles there have been earlier results, see e.g. \cite{gylss,reza}). 

For graphs $G_1,G_2,\dots,G_c$, the \textit{Ramsey number} $R(B^rG_1,B^rG_2,\dots,B^rG_c)$ is the smallest integer $N$ such that if the hyperedges of the complete $r$-uniform hypergraph on $N$ vertices are colored using colors $1,2,\dots,c$, then there is an $i\le c$, such that there is a Berge copy of $G_i$ with each of its hyperedges colored $i$. In case $G_1=G_2=\dots=G_c$, we denote the Ramsey number by $R^c(B^rG_1)$.

Most of the earlier results \cite{gmov,STWZ} focused on the case $c$ and $r$ are fixed and the graphs $G_1,G_2,\dots,G_c$ can be large. This is our approach here as well. In particular \cite{gmov} determined $R(B^rG_1,B^rG_2,\dots,B^rG_c)$ exactly in case $r\ge 2c$ and at least one of the graphs $G_i$ has at least $12c\binom{r}{2}$ vertices. 
In case $r=2c-1$, they proved $1+c\lfloor\frac{n-2}{c-1}\rfloor\le R^c(B^rK_n)\le (2c-1)\frac{n}{c-3}$, provided $n$ is large enough. Note that for $c=2$ and $r=3$, the exact result $R^2(B^3K_n)=2n-3$ if $n\ge 5$ is proved in \cite{STWZ}. As our first result, we close the gap for larger $c$, apart from a constant additive term, by showing the following.

\begin{thm}\label{thm1} For any $r$ and $c$, if $n$ is large enough, then we have \[R^c(B^{2c-1}K_n)\le\frac{cn}{c-1}+c(2c-1)(c-2)\binom{2c-1}{2}.\]

\end{thm}

We remark that it is not hard to see that the Ramsey number is monotone decreasing in $r$. It follows from the well-known fact that there is an injection from the 2-element subsets of a finite set $X$ to the $r$-element subsets of $X$ (provided $|X|\ge r+2$), such that each 2-set is contained in its image. Also, by definition, the Ramsey number monotone increases in $c$. What happens if both $r$ and $c$ increase? We show that if both increase by 1, the order of magnitude cannot increase.

\begin{thm}\label{thm2} For any graph $G$, $R^c(B^rG)\le r\binom{r}{2} R^{c-1}(B^{r-1}G)+|V(G)|$.

\end{thm}

It was shown in \cite{gmov} that if $r>c$, then the Ramsey number is linear in the number of vertices $n$. However, for $r=c=3$, \cite{STWZ} showed a super-linear lower bound $\Omega(n^2/\log n)$. Moreover, they showed for every $r$ a super-linear lower bound, in case $c$ is large enough. Note that their bounds are sub-quadratic. For the case $r=3$, a polynomial lower bound with degree increasing with $c$ was given in \cite{gmov}. Here we give a similar bound for any $r$.

\begin{thm}\label{affin} 
\textbf{(i)} Let $c>\binom{r}{2}$ and $n$ be large enough. 
Then $R^c(B^rK_n)\ge (n-1)^2+1$.

\textbf{(ii)} Let $c>(d-1)\binom{r}{2}$. Then $R^c(B^rK_n)= \Omega(n^d)$.
\end{thm}

In case $r=3$, this bound with $d=2$ and $c=4$ improves the best known lower bound $\Omega(n^2/\log n)$ to $\Omega(n^2)$. In \cite{gmov} a lower bound $R^c(B^3K_n)= \Omega\left(\frac{n^{\lceil\frac{c}{2}\rceil}}{(\log n)^{\lceil c/2\rceil-1}}\right)$ was given for every $c$. The proof was by induction on $c$ from $i$ to $i+2$, hence our improvement on $c=4$ gives an improvement for every even $c\ge 4$.

The above result leaves open the cases $r\le c\le \binom{r}{2}$. Here we show that the Ramsey number of the Berge clique is super-linear in these cases.

\begin{thm}\label{konstru} $R^r(B^rK_n)=\Omega(n^{1+1/(r-2)}/\log n)$.

\end{thm}

We present the proofs of the above theorems in Section 2 and finish the paper with some concluding remarks in Section 3.

\section{Proofs}

Given a colored complete $r$-uniform hypergraph on $N$ vertices, we will consider the complete graph on the same vertex set and call it the (colored) shadow graph. We will say that an edge is \textit{light in color $i$} if it is contained in less than $\binom{r}{2}$ hyperedges of color $i$, and \textit{heavy in color $i$} otherwise. Let $E_i$ denote the set of edges that are light in color $i$. We will use the following simple lemma from \cite{gmov}.

\begin{lemma}\label{trivi} If there is a copy of $G$ in the shadow graph such that all of its edges are heavy in color $i$, then there is a Berge-$G$ of color $i$ in the complete hypergraph.

\end{lemma}

Let $V_i$ denote the set of vertices incident to at least one edge of $E_i$. Now we are ready to prove 
Theorem \ref{thm1}. Recall that it claims that $R^c(B^{2c-1}K_n)\le\frac{cn}{c-1}+c(2c-1)(c-2)\binom{2c-1}{2}$ if $n$ is large enough.

\begin{proof}[Proof of Theorem \ref{thm1}] We let $r=2c-1$, and we will refer to it as $r$ whenever we talk about uniformity, or about a constant arising from the uniformity. For constants arising from the uniformity, we use $c$. We hope it helps follow the argument.

Let us consider a $c$-coloring of the complete $r$-uniform hypergraph on $N$ vertices without a monochromatic Berge-$K_n$, and assume $N>\frac{cn}{c-1}+cr(c-2)\binom{r}{2}$. First we show $|V_i|>\frac{N}{c}+(c-1)r(c-2)\binom{r}{2}$ for every $i\le c$. Indeed, otherwise we can consider the other $N-|V_i|\ge n$ vertices. Every edge inside that set is heavy in color $i$, thus we can find a monochromatic Berge-$K_n$ there, using Lemma \ref{trivi}. Note that it also implies $|E_i|>N/2c$.

Next we will show that $|V_c\cap V_{c-1}|\le r(c-2)\binom{r}{2}$. Assume otherwise and pick $x=r(c-2)\binom{r}{2}+1$ vertices $v_1,\dots,v_x$ from the intersection, and for each $v_i$ let $u_iv_i\in E_{c-1}$ and $w_iv_i\in E_c$ arbitrary. Let us consider the following colored hypergraph $\cH$. It consists of the $r$-edges that contain $u_i,v_i,w_i$ for some $i\le x$, and contain an edge from every $E_i$. Note that the union of $u_i,v_i,w_i$ and an edge from every $E_i$ has size at most $r$, thus there is a hyperedge in $\cH$ containing them.


Let us pick an arbitrary edge $e_j\in E_j$ for every $j\le c-2$. Consider the hyperedges of $\cH$ containing each of them. At most $(c-2)\binom{r}{2}$ of them has color $1,2,\dots, c-2$. Each of those contains $u_i,v_i,w_i$ for at most $r$ different $i\le x$ (as the vertices $v_i$ are pairwise distinct, each hyperedge contains $v_i$ for at most $r$ different $i\le x$). 
As $x>r(c-2)\binom{r}{2}$, there is a hyperedge of $\cH$ of color $c-1$ or $c$ that contains $e_1$ and $u_i,v_i,w_i$ for some $i$. 

Let us consider now the subhypergraph $\cH'$ of $\cH$ consisting of hyperedges of color $c-1$ or $c$ that contain $u_i,v_i,w_i$ for some $i$. By the above, for every $e_1\in E_1$ there is a hyperedge in $\cH'$ containing it. As a hyperedge contains at most $\binom{r}{2}$ edges of $E_1$, it means $\cH'$ contains at least $|E_1|/\binom{r}{2}$ hyperedges. On the other hand, there are at most $2\binom{r}{2}$ hyperedges of color $c-1$ or $c$ for every $i\le x$, showing $|E_1|/\binom{r}{2}\le 2x\binom{r}{2}$, a contradiction with $n$ (thus $N$) being large enough.

We obtained that $|V_i\cap V_j|\le r(c-2)\binom{r}{2}$ for any $i<j\le c$, by symmetry.
Now we have $N\ge \sum_{i=1}^c |V_i|-\sum_{i<j\le c}|V_i\cap V_j|$. This implies $\sum_{i=1}^c |V_i|\le N+\binom{c}{2}r(c-2)\binom{r}{2}$, hence there is an $i\le c$ with $|V_i|\le \frac{N}{c}+(c-1)r(c-2)\binom{r}{2}$. This is a contradiction, finishing the proof.

\end{proof}

We continue with the proof of Theorem \ref{thm2}, which states $R^c(B^rG)\le r\binom{r}{2} R^{c-1}(B^{r-1}G)+|V(G)|$.

\begin{proof}[Proof of Theorem \ref{thm2}] Let us consider a $c$-coloring of the complete $r$-uniform hypergraph on $N=R^c(B^rG)-1$ vertices without a monochromatic Berge-$K_n$. 

Let us consider an edge $u_1v_1\in E_1$ and let $U_1$ be the set of vertices contained in hyperedges of color 1 together with $u_1$ and $v_1$. Thus we have $|U_1|\le \left(\binom{r}{2}-1\right)(r-2)$. Let us delete $u_1$, $v_1$ and $U_1$, and pick $u_2v_2\in E_1$ from the remaining set of vertices. In general, after deleting $u_1,\dots,u_i$, $v_1,\dots,v_i$, $U_1,\dots,U_i$, we pick $u_{i+1}v_{i+1}\in E_1$ and let $U_{i+1}$ be the set of at most $(\binom{r}{2}-1)(r-2)$ vertices contained in hyperedges of color 1 together with $u_{i+1}$ and $v_{i+1}$. As each time we delete at most $(\binom{r}{2}-1)(r-2)+2$ vertices, we can pick $v_m$ with $m=R^{c-1}(B^{r-1}G)$. 

Let us consider the $(r-1)$-uniform complete hypergraph with vertex set $V=\{v_1,\dots,v_m\}$. We will give a $(c-1)$-coloring of it. Let $R=\{v_{i_1},\dots,v_{i_{r-1}}\}$ be a hyperedge of it with $i_1< i_2<\dots<i_{r-1}$. Then the color of this hyperedge should be the color of the $r$-edge $R\cup \{u_{i_1}\}$ in the original hyperedge. As we deleted $U_{i_1}$ before picking the other vertices, this color is not $1$. Hence we colored our $(r-1)$-uniform complete hypergraph on $m=R^{c-1}(B^{r-1}G)$ vertices with $c-1$ colors.

By definition, there is a monochromatic Berge-$G$ in this hypergraph. Each hyperedge of it inherited its color from an $r$-edge of the original hypergraph containing it, and those $r$-edges are pairwise distinct. This implies those $r$-sets form a monochromatic Berge-$G$ in the original complete $r$-uniform hypergraph, a contradiction.






\end{proof}

Before our next proof, we need to define \textit{$k$-nets}. A $k$-net of order $n$ is an incidence structure with $n^2$ vertices and $nk$ sets of size $n$, called lines, such that two lines share at most one vertex and there are $k$ ways to partition the vertex set into $n$ parallel (i.e. disjoint) lines. It is well-known that there exists a $k$-net of order $n$ if $n>n_0$. The best known upper bound on $n_0$ is roughly $k^{14.3}$ \cite{Lu} (note that a $k$-net is equivalent to a set of $l-2$ mutually orthogonal latin squares).

Now we are ready to prove Theorem \ref{affin}, that we restate here for convenience.
\begin{thm*} 
\textbf{(i)} Let $c>\binom{r}{2}$ and $n$ be large enough. 
Then $R^c(B^rK_n)\ge (n-1)^2+1$.

\textbf{(ii)} Let $c>(d-1)\binom{r}{2}$. Then $R^c(B^rK_n)= \Omega(n^d)$.
\end{thm*}

\begin{proof}[Proof of Theorem \ref{affin}] 

To prove \textbf{(i)}, consider a $c$-net of order $n-1$. 
For each color, we assign a class of parallel lines. Then for each $r$-set $R$, the $\binom{r}{2}$ pairs of vertices contained in $R$ define at most $\binom{r}{2}$ lines, thus there are at most $\binom{r}{2}$ classes of parallel lines such that two vertices in $R$ belong to one of the lines. Thus there is at least one color, say blue, such that each blue line contains only one vertex from $R$. We color $R$ blue (in case there are more available colors, we pick arbitrarily).

That means every blue hyperedge contains at most one vertex from a blue line, thus every monoblue Berge-clique contains at most one vertex from each blue line. Hence the largest monoblue clique has size at most $n-1$, and this holds for every color, finishing the proof.

To prove \textbf{(ii)}, assume first that $n-1$ is a prime power and consider an affine space of order $n-1$ and dimension $d$. For each color, we assign a class of parallel hyperplanes. Two points of an $r$-set $R$ define a line, thus they are shared by $d-1$ hyperplanes. Hence there are at most $(d-1)\binom{r}{2}$ classes of parallel hyperplanes such that two vertices in $R$ belong to one of those hyperplanes. Similarly to \textbf{(i)}, we can pick an (arbitrary) color not assigned to those classes, and obtain that a monochromatic clique shares at most one vertex with each hyperplanes in the corresponding class, hence the monochromatic clique has at most $n-1$ vertices.

This gives the lower bound $(n-1)^d+1$ if a $d$-dimensional affine space of order $n-1$ exists, in particular if $n-1$ is a prime power. It is well-known that there is a prime $p$ with $(n-1)/2\le p \le n-1$, using $p$ as the order of the affine space we obtain the lower bound $(\frac{n-1}{2})^d=\Omega(n^d)$.

\end{proof}

Instead of proving Theorem \ref{konstru}, we will prove the following more general theorem.

\begin{thm}\label{konstru2}
For every $r$, if $n$ is large enough, there exists a constant $c_r$ such that we have the following. Let $X$ be a set of size $\lfloor 2^{n/2}\rfloor$ if $r=2$ and $\lfloor c_r\frac{n^{1+1/(r-2)}}{\log n}\rfloor$ if $r\ge 3$. For non-empty subsets $T$ of size at most $r$, we can assign a subset of $r$ colors $S(T)$ the following way.

$\bullet_1$ If $T\subset T'$, then $S(T')\subset S(T)$.

$\bullet_2$ If $T\subset X$ is of size $t\le r$, then $S(T)$ has size $r-t+1$, and for every color $s\in S(T)$ there are at least $\binom{r}{2}$ sets $T'$ of size $t+1$ with $S(T')=S(T)\setminus \{s\}$.

$\bullet_3$ For every color $i$, the graph $G_i$ consisting of the 2-subsets $T$ with $i\in S(T)$ is $K_n$-free.

\end{thm}

The above theorem implies Theorem \ref{konstru}, as the $r$-sets all have exactly one color, and every subedge of a hyperedge of color $i$ has $i$ among its colors, thus there is no $K_n$ among them by $\bullet_3$. This implies there is no Berge-$K_n$ in color $i$. The other properties are needed only for the induction on $r$.

\begin{proof}
We prove the theorem by induction on $r$. For the base case $r=2$ Erd\H os \cite{Erdos} showed that there is a blue-red coloring of a graph on $\lfloor2^{n/2}\rfloor$ vertices without a monochromatic $K_n$. We assign both colors to the singletons, and we only need that every vertex is contained in at least one edge in both colors. It is easy to see that the well-known construction of Erd\H os satisfies this property.

Let us continue with the induction step, and apply the induction hypothesis with $r-1$ in place of $r$ and $n'=\lfloor c'_{r}(n) n^{\frac{r-3}{r-2}}\log n\rfloor$ in place of $n$, for some $c'_r(n)$ chosen later. Assume first $r>3$. Then we have a set $X'$ of size 
\[\left\lfloor c_{r-1}\frac{n'{^{1+1/(r-3)}}}{\log n'}\right\rfloor=\frac{c_{r-1}c_r'(n){^{\frac{r-2}{r-3}}}n\frac{r-3}{r-2}\log n}{\frac{r-3}{r-2}\log n+\log c_r'+\log\log n}= n-1,\]
for a properly chosen $c_r'(n)$, and an assignments of subsets of $r-1$ colors satisfying $\bullet_1$, $\bullet_2$ and $\bullet_3$. Note that $1/\sqrt{c_{r-1}}\le c_r'(n)\le 2/c_{r-1}$, thus in further calculations we consider $c_r'(n)$ as a constant.

In case $r=3$, a similar calculation shows that we can again find $X'$ of size $n-1$ and assignments $S'$ of subsets of two colors satisfying $\bullet_1$, $\bullet_2$ and $\bullet_3$. We leave the details to the interested Reader.

We take $\lfloor n^{1/(r-2)}\sqrt{c_{r-1}}/\log n\rfloor -1\le \lfloor n^{1/(r-2)}/c_r'(n)\log n\rfloor -1$ copies of $X'$. This is going to be the vertex set $X$ (and this defines $c_r$). Now we describe the assignment of the colors. Each set $T$ of size at most $r$ inside a copy of $X'$ will keep the $r-|T|$ colors already assigned to it (i.e. $S'(T)$), hence we need to assign one more color to each of them. We add color $r$ to $S'(T)$ for every set inside a copy to obtain $S(T)$.

Let us consider a set $T$ of size $t\le r$, and assume it intersects $k$ copies of $\cH_{r-1}$, in $t_1\ge t_2\ge \dots \ge t_k$ vertices (thus $\sum_{i=1}^kr_i=t$). Let $T_i$ be the intersection of $T$ with the $i$th copy of $\cH_{r-1}$. Then $|S'(T_i)|=r-t_i$.  If we add these numbers up for all $i\le k$, we obtain $kr-\sum_{i=1}^kt_i=kr-t$. This means that there are at least $k-1+r-t$ colors appearing $k$ times (as only $r-1$ colors are used there). We call these the \textit{available} colors for $T$. Let $S(T)$ be the set of the first $r-t+1$ available colors.

Let us go through the required properties, starting with $\bullet_1$. We assigned a set of $r-t+1$ colors to every subset of size $t\le r$. Let us consider $T\subset T'$. We can assume $|T'|=|T|+1$, as it will imply the other cases. If $T$ and $T'$ are inside a copy of $X'$, then $S(T')=S'(T')\cup \{r\}\subset S'(T)\cup \{r\}=S(T)$. If $T$ is inside a copy of $X'$ and $T'=T\cup \{v\}$ with $v$ from another copy, then $S'(T)$ is the set of available colors, thus $S(T')\subseteq S'(T)\subset S(T)$. If $T$ is not inside a copy of $X'$, let $T_i'$ be the intersection of $T'$ with the $i$th copy of $\cH_{r-1}$, thus $T_i\subset T_i'$. This implies that the set of available colors for $T'$ is a subset of the set of available colors for $T$. There is only one extra color available for $T$, thus the set of the first $r-t+1$ available colors for $T$ contains the set of the first $r-t$ available colors for $T'$.

Let us continue with $\bullet_2$. Assume first that $T$ is inside a copy of $X'$. If $s$ is not the new color $r$, then there are at least $\binom{r}{2}$ sets $T'$ of size $t+1$ with $S'(T')=S'(T)\setminus \{s\}$ for each of the $\binom{r}{2}$ sets $T'$. This implies $S(T')=S(T)\setminus \{s\}$ for these $\binom{r}{2}$ sets. If $s=r$, consider an arbitrary $v$ from another copy, and let $T'=T\cup \{v\}$. Then $S'(T)$ is the set of available colors for $T'$, thus $S(T')=S'(T)=S(T)\setminus \{r\}$. Obviously there are more than $\binom{r}{2}$ sets $T'$ that can be obtained this way. Finally, assume that $T$ is not inside a copy of $X'$. Any $s\in S(T)$ is in $S'(T_i)$ for every $i$. There are at least $\binom{r}{2}$ vertices $v$ in the $i$th copy of $\cH_{r-1}$ such that $S'(T_i\cup\{v\})=S'(T_i)-\{s\}$ by the induction hypothesis. Then $s$ is not an available color for $T\cup \{v\}$, thus it is not in $S(T\cup \{v\})$. But we have already seen that $S(T\cup \{v\})\subset S(T)$ and has size $r-t$, thus $S(T\cup \{v\})$ has to be $S(T)\setminus \{s\}$.

To see $\bullet_3$, observe that the graph consisting of the 2-subsets $T$ with $r\in S(T)$ does not have any edges between two copies of $X'$. As every copy has $n-1$ vertices, we are done. For the other colors, observe that an edge inside a copy of $X'$ got only $r$ as a new color. Thus, inside a copy, the largest clique of $G_i$ has size at most $\lfloor c'_{r}(n) n^{\frac{r-3}{r-2}}\log n\rfloor$. As there are less than $n/\lfloor c'_{r}(n) n^{\frac{r-3}{r-2}}\log n\rfloor$ copies, the largest clique in $G_i$ has less than $n$ vertices.

\end{proof}

Note that the $\log n$ in the denominator of the cardinality of $X$ in the above theorem could be improved to be $\log n^q$ for some $q<1$, but the improvement does not seem to worth the additional calculations.

\section{Concluding remarks}

The most important question regarding Berge-Ramsey problems may be the following. Is $R^c(B^rG)$ always polynomial in $n$ for $r\ge 3$? It is mentioned as a possibility in \cite{gmov}. We improved some of the lower bounds, but did not get any closer to answering this question.

As we are interested in the order of magnitude,
we did not make attempt to optimize the constant factors in Theorems \ref{thm2}, \ref{affin} and \ref{konstru}. Moreover, as we have mentioned, even the $\log$ factor could be improved in Theorem \ref{konstru2}, thus also in Theorem \ref{konstru}. As we have no reason to believe that the exponent of $n$ is close to being sharp in those statements, it seemed better to avoid lengthening the proof for a small improvement.

In Theorem \ref{affin}, the basic idea is to use an affine space of order $n-1$. It does not necessarily exist, but it exists if $n-1$ is a prime power. The set of prime powers is dense enough so that we obtain the desired result for the order of magnitude. However, we lose a large constant factor (exponential in $d$). In \textbf{(i)} we used $k$-nets to get rid of that factor. An equivalent of $k$-nets in higher dimension would help in the general case.

\end{document}